\documentclass[12pt]{article}
\usepackage[utf8]{inputenc}
\usepackage{amsmath}
\usepackage{amssymb}
\usepackage{mathrsfs}
\usepackage{amsthm}
\usepackage{cite}
\usepackage{hyperref}
\usepackage{enumerate}
\usepackage{combelow}
\usepackage[left=0.7in,right=0.7in,top=0.7in,bottom=0.7in]{geometry}
\usepackage{titlesec}
\titleformat*{\section}{\large\bfseries}
\newtheorem{theorem}{Theorem}[section]

\newtheorem{definition}[theorem]{Definition}
\newtheorem{example}[theorem]{Example}

\newtheorem{remark}[theorem]{Remark}
\numberwithin{equation}{section}
\title{A new approach to coincidence and common fixed points under a homotopy of families of mappings in $b$-metric spaces}
\author{\large Anuradha Gupta and Manu Rohilla$^*$}
\date{}
\begin{document}
\maketitle
\begin{abstract}
In this paper we derive coincidence and common fixed point results under order homotopies of families of mappings  in preordered $b$-metric spaces.
\end{abstract}   
\section{Introduction and Preliminaries}
Throughout this paper, $\mathbb{N}$ denotes the set of natural numbers. Bakhtin \cite{3} and Czerwik \cite{4} generalized the concept of metric spaces and introduced the notion of $b$-metric spaces as follows:
\begin{definition}
\emph{Let $X$ be a nonempty set. A $b$-metric is a function $d: X \times X \rightarrow [0, \infty)$ such that for all $x,y,z \in X$ and a constant $s \geq 1$, the following conditions are satisfied:}

\emph{(i) $d(x,y)=0$ if and only if $x=y$, }

\emph{(ii) $d(x,y)=d(y,x)$,}

\emph{(iii) $d(x,y) \leq s[d(x,z)+d(z,y)]$.}\\
\emph{The pair $(X,d)$ is called a $b$-metric space and the number $s$ is called the coefficient of $(X,d)$. If we take $s=1$, then $(X,d)$ reduces to a metric space.}
\end{definition} 
Let $X$ be a nonempty set equipped with a binary relation $\preccurlyeq$. Then the relation $\preccurlyeq$ is

(i) \emph{reflexive}: if $x \preccurlyeq x$ for all $x \in X$,

(ii) \emph{antisymmetric}: if $x \preccurlyeq y$ and $y \preccurlyeq x$, then $x=y$ for all $x,y \in X$,

(iii) \emph{transitive}: if $x \preccurlyeq y$ and $y \preccurlyeq z$, then $x \preccurlyeq z$ for all $x,y,z \in X$.

A binary relation $\preccurlyeq$ is said to be a \emph{preorder} if it is reflexive and transitive. A preorder is said to be a \emph{partial order} if it is antisymmetric. Let $(X,d,\preccurlyeq)$ be a $b$-metric space with coefficient $s \geq 1$ equipped with a binary relation $\preccurlyeq$.  Motivated by Kamihigashi and Stachurski \cite{7} we define the following axiom:

 $d$ is $s$-\emph{regular} if for all $x,y,z \in $ we have
$$x \preccurlyeq y \preccurlyeq z \quad \mbox{implies} \quad \max\{d(x,y),d(y,z)\} \leq s^2 d(x,z).$$

Evidently, if $d$ is regular in the sense of Kamihigashi and Stachurski \cite{7}, then $d$ is $s$-regular but the converse is not true as illustrated by the following example: 

Let $X=[0,\infty)$ and define $d:X \times X \rightarrow [0, \infty)$ by $$d(x,y)=\left\{\begin{array}{cc}
(x+y)^2,& \mbox{if}\thinspace \thinspace x \neq y,\\
0, &\mbox{if} \thinspace \thinspace x=y.
\end{array}
\right.$$
Then $(X,d)$ is a $b$-metric space with coefficient $s=2$. Suppose that $X$ is equipped with a preorder relation defined by: for $x,y \in X$, $x \preccurlyeq y$ if and only if $x \leq y$. It can be easily observed that if $x \preccurlyeq y \preccurlyeq z$, then $\max\{d(x,y),d(y,z)\} \leq 4 d(x,z)$ for all $x,y,z \in X$. This implies that $d$ is $s$-regular. Also, $1 \preccurlyeq 2 \preccurlyeq 3$ but $\max \{d(1,2),d(2,3)\}=25 \nleq d(1,3)$. Therefore, $d$ is not regular.

 Kamihigashi and Stachurski \cite{7} and Batsari et al. \cite{2} established fixed point results of order-preserving self mappings in spaces equipped with a transitive binary relation and some distance measures. Recently, Fomenko and Podoprikhin \cite{5,9} generalized the results of Arutyunov \cite{11,1} to the case of families of multivalued mappings in partially ordered sets. They \cite{6,9} introduced the notion of order homotopy and proved that common fixed point and coincidence point properties are preserved under homotopies of  families of mappings in ordered sets. 

In this paper we establish results on preservation of  coincidence points and common fixed points under order homotopies  of families of mappings. The main objective of the paper is to develop a new approach  and obtain sufficient conditions to gurantee the existence of coincidence and common fixed points under order homotopies of families of mappings in preordered $s$-regular $b$-metric spaces. We have generalized the results of \cite{6,8} to the case of  preordered $s$-regular $b$-metric spaces and preordered regular metric spaces. This approach has enabled us to  replace the classical approach of considering the underlying  space to be complete by a preordered space and $b$-metric  function satisfying the axiom of $s$-regularity. As a consequence we establish coincidence and common fixed points results under order homotopic families of mappings in the context of metric spaces and $b$-metric spaces from a unique point of view. An example demonstrating the utility of the results obtained is also provided.
\section{Coincidence point results}
In this section we prove some coincidence point results under a homotopy of families of mappings in preordered $s$-regular $b$-metric spaces. As a consequence, we formulate coincidence and fixed point results  in preordered regular  metric spaces.

For formulating the results, we recall some definitions and notions introduced by Fomenko and Podoprikhin \cite{6,8}.
\begin{definition}
\emph{Let $(X,\preccurlyeq)$ be a preordered set. A mapping $T:X \rightarrow X$} covers (covers from above) \emph{a mapping $S:X \rightarrow X$ on a set $A \subset X$ if and only if for any $x \in A$ such that $S(x) \preccurlyeq T(x)$ ($S(x) \succcurlyeq T(x)$) then there exists an element $y \in X$ satisfying $y\preccurlyeq x$ ($y \succcurlyeq x$) and $T(y)=S(x)$. If $A=X$, then we say that $T$} covers (covers from above) \emph{$S$.}
\end{definition}
Let $(X,\preccurlyeq)$ be a preordered set. A mapping $T:X \rightarrow X$ is said to be \emph{isotone} if for all $x, y \in X$
$$x\preccurlyeq y \quad \mbox{implies} \quad T(x) \preccurlyeq T(y).$$

Let $T,S : X \rightarrow X$ be a pair of self mappings on a preordered set $(X,\preccurlyeq)$. Let $\mathcal{C}(T,S,\preccurlyeq)$ denote the set of chains $C$ such that for all $x,y \in C$ we have

(i) $T(x) \preccurlyeq S(x)$,

(ii) $x \prec y$ implies $S(x) \preccurlyeq T(y)$, 

(iii) $S(C) \subset T(X)$.

Let $\mathcal{C}^*(T,S,\preccurlyeq)$ denote the set of chains $C$ such that for all $x,y \in C$ we have

(i) $T(x) \succcurlyeq S(x)$,

(ii) $x \prec y$ implies $T(x) \preccurlyeq S(y)$, 

(iii) $S(C) \subset T(X)$.

Recall that the dual order of $\preccurlyeq$ is the order $\preccurlyeq^*$ defined by: $x \preccurlyeq y$ if and only if $y \preccurlyeq^* x$. It can be easily observed that $\mathcal{C}^*(T,S,\preccurlyeq)=\mathcal{C}(T,S,\preccurlyeq^*)$. For a fixed element $x_0 \in X$ define the set
\begin{align*}
O_X(x_0)&:= \{x \in X: x\preccurlyeq x_0\},\\
O_X^*(x_0)&:=\{x \in X: x \succcurlyeq x_0\},\\
\mathcal{C}(x_0,T,S,\preccurlyeq)&:=\{C \in \mathcal{C}(T,S,\preccurlyeq):C \subset O_X(x_0) \thinspace \thinspace \mbox{and} \thinspace \thinspace S(C) \subset T(O_X(x_0)) \},\\
\mathcal{C}^*(x_0,T,S,\preccurlyeq)&:=\{C \in \mathcal{C}^*(T,S,\preccurlyeq):C \subset O_X^*(x_0) \thinspace \thinspace \mbox{and} \thinspace \thinspace S(C) \subset T(O_X^*(x_0)) \},\\
\mbox{Coin}(T,S)&:=\{x \in X: Tx=Sx\},\\
\mbox{Fix}(T)&:=\{ x \in X: Tx=x \}.
\end{align*}

Recall that an element $x \in X$ is said to a \emph{minimal} (\emph{maximal}) element of $X$ if and only if for every other element $y \in X$ we have $x\preccurlyeq y$ ($x \succcurlyeq y$).
\begin{theorem}\label{theorem1}
Let $(X,d,\preccurlyeq)$ be a preordered $s$-regular $b$-metric space. Let $T,S:X \rightarrow X$ be self mappings and there exists $x_0 \in X$ such that $T(x_0) \preccurlyeq S(x_0)$. Suppose that the following conditions are satisfied:

(i) the mapping $T$ is isotone,

(ii) the mapping $S$ covers the mapping $T$ on the set $O_X(x_0)$,

(iii) any chain $C \in \mathcal{C}(x_0,T,S,\preccurlyeq)$ has a lower bound $w \in X$ satisfying $w \preccurlyeq T(w)$ and there exists $z \in X$ such that 
\begin{align*}
&S^i(z) \preccurlyeq S(w) \preccurlyeq T(w) \quad \mbox{for all} \thinspace \thinspace \thinspace  i \in \mathbb{N},\\
&d(T^i(w),S^i(z)) \rightarrow 0 \quad \mbox{as} \thinspace \thinspace \thinspace i \rightarrow \infty.
\end{align*}
Then the set $ \mbox{Coin}(T,S) \cap O_X(x_0)$ is nonempty and contains a minimal element.
\end{theorem}
\begin{proof}
The proof is divided into four steps where the existence of an element in $ \mbox{Coin}(T,S) \cap O_X(x_0)$ is proven in Steps $1$ to Step $3$ and its minimality is established in Step $4$. \\
\textbf{Step 1} Set 
$$V=\{x \in O_X(x_0): T(x)\preccurlyeq S(x), \thinspace Sx \in T(O_X(x_0)) \}.$$
We show that $V$ is a nonempty set. Since $\preccurlyeq$ is reflexive, $x_0 \in O_X(x_0)$. Also, $T(x_0) \preccurlyeq S(x_0)$. By (ii) there exists $x_1 \in X$ satisfying 
$$x_1 \preccurlyeq x_0 \quad \mbox{and} \quad T(x_0)=S(x_1).$$
Since $T$ is isotone, $T(x_1) \preccurlyeq S(x_1)$. Also, $S(x_1)=T(x_0) \in T(O_X(x_0))$ which implies that $x_1 \in V$. Define an ordered relation on $V$ as follows: $v_1 \sqsubseteq v_2$ if either $v_1=v_2$ or $v_1 \prec v_2$ and $S(v_1) \preccurlyeq T(v_2)$. It is easily seen that $\sqsubseteq$ is reflexive and antisymmetric. Suppose that $v_1,v_2,v_3 \in V$, $v_1 \sqsubseteq v_2$ and $v_2 \sqsubseteq v_3$. If $v_1=v_2$ or $v_2=v_3$, then evidently, $v_1 \sqsubseteq v_3$. If $v_1 \prec v_2$, $v_2 \prec v_3$, $S(v_1) \preccurlyeq T(v_2)$ and $S(v_2) \preccurlyeq T(v_3)$, then $v_1 \prec v_3$. Also, $v_2 \in V$ gives $T(v_2) \preccurlyeq S(v_2)$ which implies that $S(v_1) \preccurlyeq T(v_3)$. Therefore, $v_1 \sqsubseteq v_3$ which gives $\sqsubseteq$ is a partial order on $V$. By Hausdorff maximal principle, there exists a maximal chain $C \subset V$ with respect to the order $\sqsubseteq$.\\
\textbf{Step 2} We show that $C \in \mathcal{C}(x_0,T,S,\preccurlyeq)$. By the definition of V we have $C \subset  O_X(x_0)$, $S(C) \subset T(O_X(x_0))$ and $T(x) \preccurlyeq S(x)$ for all $x \in C$. Let $x,y \in C$ and $x \prec y$. Since $C$ is a chain, $x \sqsubseteq y$ which gives $S(x) \preccurlyeq T(y)$. Therefore, $C \in \mathcal{C}(x_0,T,S,\preccurlyeq)$. By (iii), there exists a lower bound $w \in X$ of $C$ with respect to the order $\preccurlyeq$.\\
\textbf{Step 3} It is to be shown  that $w \in  \mbox{Coin}(T,S) \cap O_X(x_0)$.  Since $w$ is a lower bound of $C$, $w \preccurlyeq x$ for all $x \in C$. Also, $C \subset O_X(x_0)$ and $\preccurlyeq$ is transitive. Therefore, $w \preccurlyeq x_0$ which implies that $w \in O_X(x_0)$. Consider
\begin{align}\label{equation1}
d(T(w),S(w)) & \leq sd(T(w),T^i(w))+sd(T^i(w),S(w)) \nonumber \\
&\leq s^2 d(T(w),S^i(z))+2s^2d(T^i(w),S^i(z))+s^2d(S^i(z),S(w)).
\end{align}
Since $S^i(z) \preccurlyeq S(w) \preccurlyeq T(w)$ for all $i \in \mathbb{N}$ and $d$ is $s$-regular, $d(S^i(z),S(w)) \leq s^2 d(S^i(z),T(w))$. Therefore, (\ref{equation1}) becomes
\begin{align}\label{equation2}
d(T(w),S(w)) \leq (s^4+s^2)d(S^i(z),T(w))+2s^2d(T^i(w),S^i(z)).
\end{align} Also, $w \preccurlyeq T(w)$ and $T$ is isotone gives $T(w) \preccurlyeq T^2(w)$. Continuing likewise we get $w \preccurlyeq T(w) \preccurlyeq T^i(w)$ for all $i \in \mathbb{N}$. Therefore, $S^i(z) \preccurlyeq T(w) \preccurlyeq T^i(w)$ which gives $d(S^i(z),T(w)) \leq s^2 d(S^i(z),T^i(w))$. Using (\ref{equation2}) we have
\begin{align*}
d(T(w),S(w)) \leq (s^6+s^4+2s^2)d(T^i(w),S^i(z)).
\end{align*}
Letting $i \rightarrow \infty$ we get $T(w)=S(w)$. Therefore, $w \in  \mbox{Coin}(T,S) \cap O_X(x_0)$.\\
\textbf{Step 4} We claim that $w$ is a minimal element of the set $ \mbox{Coin}(T,S) \cap O_X(x_0)$. Assume on the contrary, there exists $u \in \mbox{Coin}(T,S) \cap O_X(x_0) $ satisfying $u \prec w$. Since $T$ is isotone, $T(u) \preccurlyeq T(w)$ which implies that $S(u) \preccurlyeq T(w)$. Therefore, $u \sqsubseteq w$. As $w$ is a lower bound of $C$, $w \preccurlyeq x$ for all $x \in C$. Therefore, $T(w) \preccurlyeq T(x)$ which implies that $S(w) \preccurlyeq T(w)$. This gives $u \sqsubseteq w \sqsubseteq x$. As $C$ is a maximal chain with respect to $\sqsubseteq$, $u \in C$ which gives $w \preccurlyeq u$, a contradiction. Hence, $w$ is a minimal element of the set $ \mbox{Coin}(T,S) \cap O_X(x_0)$.
\end{proof}
The dual version of Theorem \ref{theorem1} can be stated as follows:
\begin{theorem}\label{theorem3}
Let $(X,d,\preccurlyeq)$ be a preordered $s$-regular $b$-metric space. Let $T,S:X \rightarrow X$ be self mappings and there exists $x_0 \in X$ such that $T(x_0) \succcurlyeq S(x_0)$. Suppose that the following conditions hold:

(i) the mapping $T$ is isotone,

(ii) the mapping $S$ covers the mapping $T$ from above on the set $O_X^{*}(x_0)$,

(iii) any chain $C \in \mathcal{C}^*(x_0,T,S,\preccurlyeq)$ has an upper bound $w \in X$ satisfying $w \succcurlyeq T(w)$ and there exists $z \in X$ such that 
\begin{align*}
&S^i(z) \succcurlyeq S(w) \succcurlyeq T(w) \quad \mbox{for all} \thinspace \thinspace i \in \mathbb{N},\\
&d(T^i(w),S^i(z)) \rightarrow 0 \quad \mbox{as} \thinspace \thinspace \thinspace i \rightarrow \infty.
\end{align*}
Then the set $\mbox{Coin}(T,S) \cap O_X^*(x_0)$ is nonempty and contains a maximal element.
\end{theorem}
Podoprikhin and Fomenko \cite{8} generalized the notion of order homotopy of isotone mappings introduced by Walker \cite{10} as follows:

Let $(X,\preccurlyeq)$ be a preordered set. A pair of  mappings $T,S:X \rightarrow X$ are said to be \emph{order homotopic} if there exists a finite sequence of mappings $H_i:X \rightarrow X$ $(i=1,2,\ldots,n)$ such that
$$T=H_0\preccurlyeq H_1 \succcurlyeq H_2 \preccurlyeq \ldots \succcurlyeq H_n=S,$$
where $H_i \preccurlyeq H_j$ if and only if $H_i(x) \preccurlyeq H_j (x)$ for all $x \in X$.
\begin{theorem}\label{theorem2}
Let $(X,\preccurlyeq)$ be a preordered $s$-regular $b$-metric space. Let $(T,S)$ and $(\Tilde{T},\Tilde{S})$ be a pair of self mappings on $X$. Suppose that there are homotopies $\{H_t\}_{0 \leq t \leq n}$ and $\{K_t\}_{0 \leq t \leq n}$ between these pair of mappings such that
\begin{align*}
T=H_0 \preccurlyeq H_1 \succcurlyeq H_2 \preccurlyeq \ldots \succcurlyeq H_n=\Tilde{T},\\
S=K_0 \succcurlyeq K_1 \preccurlyeq K_2 \succcurlyeq \ldots \preccurlyeq H_n=\Tilde{S}.
\end{align*}
Suppose that $x_0 \in \mbox{Coin}(T,S)$ and the following conditions are satisfied:

(i) for each odd $t$, $1 \leq t \leq n$, $K_t$ is isotone, the mapping $H_t$ covers the mapping $K_t$ and any chain $C \in \mathcal{C}(x_0,K_t,H_t,\preccurlyeq)$ has a lower bound $w \in X$ satisfying $w \preccurlyeq K_t(w)$ and there exists $z \in X$ such that for all $i \in \mathbb{N}$ we have $$H_t^i(z) \preccurlyeq H_t(w) \preccurlyeq K_t(w) \quad \mbox{and} \quad d(K_t^i(w),H_t^i(z)) \rightarrow \infty,$$

(ii) for each even $t$, $1 \leq t \leq n$, $H_t$ is isotone, the mapping $K_t$ covers the mapping $H_t$ and any chain $C' \in \mathcal{C}(x_0,H_t,K_t,\preccurlyeq)$ has a lower bound $w' \in X$ satisfying $w' \preccurlyeq H_t(w')$ and there exists $z' \in X$ such that for all $i \in \mathbb{N}$ we have $$K_t^i(z') \preccurlyeq K_t(w') \preccurlyeq H_t(w') \quad  \mbox{and} \quad d(H_t^i(w'),K_t^i(z')) \rightarrow \infty.$$
Then there exists a chain
$$x_0\succcurlyeq x_1 \succcurlyeq x_2 \succcurlyeq \ldots  \succcurlyeq x_n$$
such that for every $t$, $1 \leq t \leq n$, $x_t \in  \mbox{Coin}(H_t,K_t) \cap O_X(x_{t-1})$ and $x_t$ is a minimal element of the set $\mbox{Coin}(H_t,K_t) \cap O_X(x_{t-1})$.
\end{theorem}
\begin{proof}
Since $x_0 \in \mbox{Coin}(H_0,K_0)$, $K_1(x_0) \preccurlyeq K_0 (x_0)=H_0(x_0) \preccurlyeq H_1(x_0)$ which implies that $K_1(x_0) \preccurlyeq H_1(x_0)$. Using (i) we infer that all the conditions of Theorem \ref{theorem1} are satisfied. Therefore, there exists $x_1 \in  \mbox{Coin}(H_1,K_1)\cap O_X(x_0)$ such that $x_1$ is a minimal element of the set $ \mbox{Coin}(H_1,K_1) \cap O_X(x_0)$. Consider $H_2(x_1) \preccurlyeq H_1(x_1)=K_1(x_1) \preccurlyeq K_2(x_1)$ which gives $H_2(x_1) \preccurlyeq K_2(x_1)$. As $\mathcal{C}(x_1,H_2,K_2, \preccurlyeq) \subset \mathcal{C}(x_0,H_2,K_2, \preccurlyeq)$ and using (ii) we deduce that all the conditions of Theorem \ref{theorem1} are satisfied. Therefore, there exists $x_2 \in \mbox{Coin}(H_2,K_2) \cap O_X(x_1)$ such that $x_2$ is a minimal element of the set  $ \mbox{Coin}(H_2,K_2) \cap O_X(x_1)$. Repeating this process we obtain a chain $$x_0\succcurlyeq x_1 \succcurlyeq x_2 \succcurlyeq \ldots  \succcurlyeq x_n,$$ where  $x_t \in \mbox{Coin}(H_t,K_t) \cap O_X(x_{t-1})$ and $x_t$ is a minimal element of the set $ \mbox{Coin}(H_t,K_t) \cap O_X(x_{t-1})$, $1 \leq t \leq n$.
\end{proof}

The following result is an immediate consequence of Theorem \ref{theorem1} and Theorem \ref{theorem3}:
\begin{theorem}\label{theorem4}
Let $(X,d,\preccurlyeq)$ be a preordered $s$-regular $b$-metric space. Let $(T,S)$ and $(\Tilde{T},\Tilde{S})$ be a pair of self mappings on $X$. Suppose that there are homotopies $\{H_t\}_{0 \leq t \leq n}$ and $\{K_t\}_{0 \leq t \leq n}$ between these pair of mappings such that
\begin{align*}
T=H_0 \preccurlyeq H_1 \succcurlyeq H_2 \preccurlyeq \ldots \succcurlyeq H_n=\Tilde{T},\\
S=K_0 \succcurlyeq K_1 \preccurlyeq K_2 \succcurlyeq \ldots \preccurlyeq H_n=\Tilde{S}.
\end{align*}
Suppose that the mappings $H_t$ are isotone, $1 \leq t \leq n$, $x_0 \in \mbox{Coin}(T,S)$ and the following conditions are satisfied:

(i) for each odd $t$, $1 \leq t \leq n$, the mapping $K_t$ covers the mapping $H_t$ from the above and any chain $C \in \mathcal{C}^*(H_t,K_t,\preccurlyeq)$ has an upper bound $w \in X$ satisfying $w \succcurlyeq H_t(w) $ and there exists $z \in X$ such that for all $i \in \mathbb{N}$ we have $$K_t^i(z) \succcurlyeq K_t(w) \succcurlyeq H_t(w) \quad  \mbox{and} \quad d(H_t^i(w),K_t^i(z)) \rightarrow \infty,$$

(ii) for each even $t$, $1 \leq t \leq n$, the mapping $K_t$ covers the mapping $H_t$ and any chain $C \in \mathcal{C}(H_t,K_t,\preccurlyeq)$ has a lower bound $w' \in X$ satisfying $w' \preccurlyeq H_t(w')$ and there exists $z' \in X$ such that for all $i \in \mathbb{N}$ we have $$K_t^i(z') \preccurlyeq K_t(w') \preccurlyeq H_t(w') \quad \mbox{and} \quad  d(H_t^i(w'),K_t^i(z')) \rightarrow \infty.$$
Then there exists a fence
$$x_0 \preccurlyeq x_1 \succcurlyeq x_2 \preccurlyeq \ldots  \succcurlyeq x_n$$
such that for each odd $t$, $1 \leq t \leq n$, $x_t \in  \mbox{Coin}(H_t,K_t) \cap O_X^*(x_{t-1})$ and $x_t$ is a maximal element of the set $\mbox{Coin}(H_t,K_t) \cap O_X^*(x_{t-1})$ and for each even $t$, $1 \leq t \leq n$, $x_t \in \mbox{Coin}(H_t,K_t) \cap O_X(x_{t-1})$ and $x_t$ is a minimal element of the set $ \mbox{Coin}(H_t,K_t) \cap  O_X(x_{t-1})$. 
\end{theorem}

Evidently, the identity mapping $I:X \rightarrow X$ covers (covers from above) any mapping $T:X \rightarrow X$. Therefore, we have the following fixed point result under a homotopy of self mappings:
\begin{theorem}\label{theorem5}
Let $(X,d,\preccurlyeq)$ be a preordered $s$-regular $b$-metric space. Let $T, \Tilde{T}:X \rightarrow X$ be isotone mappings and $\{H_t\}_{0\leq t \leq n}$ be homotopy between $T$ and $\Tilde{T}$ such that
$$T=H_0 \preccurlyeq H_1 \succcurlyeq H_2 \preccurlyeq \ldots \succcurlyeq H_n=\Tilde{T}.$$
Suppose that $x_0 \in \mbox{Fix}(T)$ and the following conditions are satisfied:

(i) for each odd $t$, $1 \leq t \leq n$, any chain $C \in \mathcal{C}^*(H_t,I, \preccurlyeq)$ has an upper bound $w \in X$ and there exists $z \in X$ such that for all $i \in \mathbb{N}$ we have $$z \succcurlyeq w \succcurlyeq H_t(w) \quad \mbox{and} \quad d(H_t^i(w),z) \rightarrow 0,$$

(ii) for each even $t$, $1 \leq t \leq n$, any chain $C \in \mathcal{C}(H_t,T,\preccurlyeq)$ has a lower bound $w' \in X$ and there exists $z' \in X$ such that for all $i \in \mathbb{N}$ we have $$z' \preccurlyeq w' \preccurlyeq H_t(w') \quad \mbox{and} \quad d(H_t^i(w'),z') \rightarrow 0.$$
Then there exists a fence
$$x_0 \preccurlyeq x_1 \succcurlyeq x_2 \preccurlyeq \ldots \succcurlyeq x_n$$
such that for each odd $t$, $1 \leq t \leq n$, $x_t \in \mbox{Fix}(H_t) \cap O_X^*(x_{t-1})$ and $x_t$ is a maximal element of the set $\mbox{Fix}(H_t) \cap O_X^*(x_{t-1})$ and for each even $t$, $1 \leq t \leq n$,  $x_t \in \mbox{Fix}(H_t) \cap O_X(x_{t-1})$ and $x_t$ is a minimal element of the set $\mbox{Fix}(H_t) \cap O_X(x_{t-1})$.
\end{theorem}

\begin{remark}
\emph{In Theorems \ref{theorem1}-\ref{theorem5} if we assume $(X,d,\preccurlyeq)$ is a preordered  regular metric space, then also the results hold.}
\end{remark}

 We conclude this section with the following example to illustrate the efficiency of our results: 
\begin{example}
\emph{Let $X=[0,2]$ and define $d: X \times X \rightarrow [0, \infty)$ by $d(x,y)=(x-y)^2$. Then $(X,d)$ is a $b$-metric space with coefficient $s=2$ and $d$ is $2$-regular.  Suppose that $X$ is equipped with a preorder given by: for $x,y \in X$,  $x \preccurlyeq y$ if and only if $x \leq y$. Define $T,H,\Tilde{T}: X \rightarrow X$ by
\begin{equation*}
\begin{split}
T(x)&= \frac{x^3}{5},\\\
H(x)&=\left\{\begin{array}{ll}
0.2,& \mbox{if}\thinspace \thinspace 0 \leq x \leq 0.2,\\
x, &\mbox{if} \thinspace \thinspace 0.2 \leq x \leq 1,\\
\frac{x^2+3}{4}, &\mbox{if} \thinspace \thinspace 1 \leq x \leq 2,
\end{array}
\right.\\\
\Tilde{T}(x)&=\left\{\begin{array}{ll}
\frac{x}{2}+0.1,& \mbox{if}\thinspace \thinspace 0 \leq x \leq 0.2,\\
x, &\mbox{if} \thinspace \thinspace 0.2 \leq x \leq 0.5,\\
\frac{2x+1}{4}, &\mbox{if} \thinspace \thinspace 0.5 \leq x \leq 2.
\end{array}
\right. 
\end{split}
\end{equation*}
Then $T \preccurlyeq H \succcurlyeq \Tilde{T}$. 
Clearly, $x_0=0 \in \mbox{Fix}(T)$ and $O_X^*(x_0)=[0,2]$. Also, every chain $C \in \mathcal{C}^*(H,I,\preccurlyeq)$ has an upper bound $w=1$ and there exists $z=1+\frac{1}{i} \in X$, $i \in \mathbb{N}$ such that $$ 1+\frac{1}{i} \succcurlyeq 1 \succcurlyeq H(1)$$ and $$d(H^i(w),z)=d\Big(1,1+\frac{1}{i} \Big)=\frac{1}{i^2} \rightarrow 0 \quad \mbox{as} \thinspace \thinspace \thinspace  i \rightarrow \infty.$$ Also, every chain $C \in \mathcal{C}(\Tilde{T},I,\preccurlyeq)$ has a lower bound $w'=0.2$ and $z'=0.2-\frac{i}{2i^2+4} \in X$, $i \in \mathbb{N}$ such that $$0.2-\frac{i}{2i^2+4} \preccurlyeq 0.2 \preccurlyeq \Tilde{T}(0.2)$$ and $$d(\Tilde{T}^i(w'),z')=d\Big(0.2,0.2-\frac{i}{2i^2+4}\Big)=\frac{i^2}{(2i^2+4)^2} \rightarrow 0 \quad \mbox{as} \thinspace \thinspace \thinspace  i \rightarrow \infty.$$ Then all the conditions of Theorem \ref{theorem5} are satisfied for a preordered $2$-regular $b$-metric space. Therefore, we get a fence $x_0 \preccurlyeq x_1 \succcurlyeq x_2$, where $x_0=0$, $x_1=1$ and $x_2=0.2$. It is observed that $x_1 \in  \mbox{Fix}(H) \cap O_X^*(x_0)$ and $x_1$ is a maximal element of the set $\mbox{Fix}(H) \cap O_X^*(x_0)$. Also,  $O_X(x_1)=[0,1]$, $x_2 \in \mbox{Fix}(\Tilde{T}) \cap O_X(x_1)$ and $x_2$ is a minimal element of the set $\mbox{Fix}(\Tilde{T}) \cap O_X(x_1)$.}
\end{example}
\section{Common Fixed Point Results}
In this section we prove some common fixed point results under a homotopy of families of mappings in preordered $s$-regular $b$-metric spaces.

Fomenko and Podoprikhin \cite{5} introduced the notion of concordantly isotone mappings as follows:

Let $(X,\preccurlyeq)$ be a preordered set, $\mathcal{I}$ be a nonempty set and a  family of mappings $\mathcal{F}=\{f_{\alpha}\}_{\alpha \in \mathcal{I}}$, where $f_{\alpha}:X \rightarrow X$ for all $\alpha \in \mathcal{I}$. The family $\mathcal{F}$ is said to be \emph{condordantly isotone} if for all $x,y \in X$
$$x \prec y \quad \mbox{implies} \quad f_{\alpha}(x) \preccurlyeq f_{\beta}(y) \thinspace \thinspace \thinspace \mbox{for all} \thinspace \thinspace  \alpha,\beta \in \mathcal{I}.$$

Let $\mathcal{C}_1(\mathcal{F},\preccurlyeq)$ denote the set of chains $C \subset \bigcup\limits_{\alpha \in \mathcal{I}}f_{\alpha}(X)$ such that

(i) $f_{\alpha}(x) \preccurlyeq x$ for all $x \in C$ and $\alpha \in \mathcal{I}$,

(ii) $x \prec y$ implies $x \preccurlyeq f_{\alpha}(y)$ for all $x,y \in C$ and $\alpha \in \mathcal{I}$.

Let $\mathcal{C}_1^*(\mathcal{F},\preccurlyeq)$ denote the set of chains $C \subset \bigcup\limits_{\alpha \in \mathcal{I}}f_{\alpha}(X)$ such that

(i) $f_{\alpha}(x) \succcurlyeq x$ for all $x \in C$ and $\alpha \in \mathcal{I}$,

(ii) $x \prec y$ implies $f_{\alpha}(x) \preccurlyeq y$ for all $x,y \in C$ and $\alpha \in \mathcal{I}$.

Observe that $\mathcal{C}_1^*(\mathcal{F},\preccurlyeq)= \mathcal{C}_1(\mathcal{F},\preccurlyeq^*)$. For a fixed element $x_0 \in X$ define the set
\begin{align*}
\mathcal{C}_1(x_0, \mathcal{F},\preccurlyeq)&:=\{C \in \mathcal{C}_1(\mathcal{F},\preccurlyeq):C \subset O_X(x_0) \cap \bigcup\limits_{\alpha \in \mathcal{I}}f_{\alpha}(O_X(x_0)) \},\\
\mathcal{C}_1^*(x_0, \mathcal{F},\preccurlyeq)&:=\{C \in \mathcal{C}_1^*(\mathcal{F},\preccurlyeq):C \subset O_X^*(x_0) \cap \bigcup\limits_{\alpha \in \mathcal{I}}f_{\alpha}(O_X^*(x_0)) \},\\
\mbox{ComFix}(\mathcal{F})&:= \{ x \in X: f_{\alpha}(x)=x \thinspace \thinspace \mbox{for all} \thinspace \thinspace \alpha \in \mathcal{I} \}.
\end{align*}
\begin{theorem}\label{theorem6}
Let $(X,d,\preccurlyeq)$ be a preordered $s$-regular $b$-metric space in which a point $x_0$ is fixed and $\mathcal{F}=\{f_{\alpha}\}_{\alpha \in \mathcal{I}}$ be a concordantly isotone family of mappings verifying $f_{\alpha}(x_0) \preccurlyeq x_0$ for all $\alpha \in \mathcal{I}$. Suppose that for any chain $C \in \mathcal{C}_1(x_0,\mathcal{F},\preccurlyeq)$ there exists $w \in X$ such that $w$ is a common lower bound of the chains $f_{\alpha}(C)$ for each $\alpha \in \mathcal{I}$ and there exists $z \in X$ and $\beta \in \mathcal{I}$ satisfying
\begin{align*}
&f_{\alpha}(w) \preccurlyeq w \preccurlyeq f_{\beta}^i(z),\\
&d(f_{\alpha}^i(w), f_{\beta}^i(z)) \rightarrow 0.
\end{align*}
Then the set $\mbox{ComFix}(\mathcal{F}) \cap O_X(x_0)$ is nonempty and contains a minimal element.
\end{theorem}
\begin{proof}
The proof is divided into three steps where the existence of an element in $\mbox{ComFix}(\mathcal{F}) \cap O_X(x_0)$ is estalished in Steps $1$ and $2$ and its minimality is proven in Step $3$. \\
\textbf{Step 1} We show that $\mathcal{C}_1(x_0, \mathcal{F},\preccurlyeq)$ is nonempty. Since $f_{\alpha} (x_0) \preccurlyeq x_0$, $$f_{\alpha} (x_0) \in O_X(x_0) \cap  \bigcup\limits_{ \alpha \in \mathcal{I}} f_{\alpha} (O_X(x_0)).$$ Also, $f_{\alpha}( x_0) \preccurlyeq x_0$ and $\mathcal{F}$ is concordantly isotone  implies that $f_{\beta}(f_{\alpha}(x_0)) \preccurlyeq f_{\alpha}(x_0)$ for all $\beta \in \mathcal{I}$. Therefore, $\{f_{\alpha}(x_0)\} \in \mathcal{C}_1(x_0,\mathcal{F},\preccurlyeq)$ which gives $\mathcal{C}_1(x_0,\mathcal{F},\preccurlyeq)$ is nonempty. Define an ordered relation on $\mathcal{C}_1(x_0, \mathcal{F},\preccurlyeq)$ as follows: $C_1 \trianglerighteq  C_2$ if and only if $C_1 \subset C_2$. It is easily seen that $\trianglerighteq$ is a partial order on $C_1(x_0, \mathcal{F},\preccurlyeq)$. Then by Hausdorff maximal principle, there exists a maximal chain $C$ in $\mathcal{C}_(x_0, \mathcal{F}, \preccurlyeq)$ containing $\{f_{\alpha}(x_0)\}$ with respect to the order $\trianglerighteq$. By our assumption there exists $w \in X$ such that $w$ is a common lower bound of the chains $f_{\alpha}(C)$ for each $\alpha \in \mathcal{I}$ and there exists $z \in X$ and $\beta \in \mathcal{I}$ satisfying
\begin{align*}
&f_{\alpha}(w) \preccurlyeq w \preccurlyeq f_{\beta}^i(z)\\
&d(f_{\alpha}^i(w), f_{\beta}^i(z)) \rightarrow 0.
\end{align*}
\textbf{Step 2} In this step we show that $w \in \mbox{ComFix}(\mathcal{F}) \cap O_X(x_0)$. Since $w$ is a common lower bound of the chains $f_{\alpha}(C)$ for all $\alpha \in \mathcal{I}$. Then $w \preccurlyeq f_{\alpha}(x)$ for all $x \in C$ and $\alpha \in \mathcal{I}$. As $f_{\alpha}(x_0) \in C$, 
$$w \preccurlyeq f_{\alpha}(f_{\alpha}(x_0)) \preccurlyeq f_{\alpha}(x_0) \preccurlyeq x_0.$$
Then transitivity of $\preccurlyeq$ implies that $w \in O_X(x_0)$. Consider 
\begin{equation}\label{equation3}
 d(f_{\alpha}(w),w) \leq s d(f_{\alpha}(w), f_{\beta}^i(z))+sd(f_{\beta}^i(z),w).
 \end{equation}
 Since $f_{\alpha}(w) \preccurlyeq w$ and $\mathcal{F}$ is concordantly isotone, $f_{\alpha}(f_{\alpha}(w))\preccurlyeq f_{\alpha}(w)$. Proceeding likewise we have $f_{\alpha}^i(w) \preccurlyeq f_{\alpha}(w) \preccurlyeq w$ for all $i \in \mathbb{N}$. Therefore, for all $i \in \mathbb{N}$ 
 $$f_{\alpha}^i(w) \preccurlyeq f_{\alpha}(w) \preccurlyeq f_{\beta}^i(z).$$
 As $d$ is $s$-regular, $d(f_{\alpha}(w),f_{\beta}^i(z)) \leq s^2 d(f_{\alpha}^i(w),f_{\beta}^i(z))$. Also, $f_{\alpha}^i(w) \preccurlyeq w \preccurlyeq f_{\beta}^i(z)$ and $d$ is $s$-regular gives $d(w, f_{\beta}^i(z)) \leq s^2d(f_{\alpha}^i(w),f_{\beta}^i(z))$. Therefore, (\ref{equation3}) becomes
 $$d(f_{\alpha}(w),w) \leq 2s^3 d(f_{\alpha}^i(w),f_{\beta}^i(z)).$$
 Letting $i \rightarrow \infty$ we get $f_{\alpha}(w)=w$ for all $\alpha \in \mathcal{I}$. Therefore, $w \in \mbox{ComFix}(\mathcal{F}) \cap O_X(x_0)$.\\
 \textbf{Step 3} We claim that $w$ is a minimal element of the set $\mbox{ComFix}(\mathcal{F}) \cap O_X(x_0)$. Assume on the contrary, there exists $v \in \mbox{ComFix}(\mathcal{F}) \cap O_X(x_0)$ satisfying $v \prec w$. As $v=f_{\alpha}(v) \in \bigcup\limits_{\alpha \in \mathcal{I}}f_{\alpha}(O_X(x_0))$, $v \in O_X(x_0) \cap \bigcup\limits_{\alpha \in \mathcal{I}}f_{\alpha}(O_X(x_0))$. Since $\preccurlyeq$ is reflexive, $f_{\alpha}(v) \preccurlyeq v$ for all $\alpha \in \mathcal{I}$. Therefore, $f_{\alpha}(x) \preccurlyeq x$ for all $x \in C \cup \{v\}$ and $\alpha \in \mathcal{I}$. As $w$ is a common lower bound of the chains $f_{\alpha}(C)$ and $v \prec w$, $v \preccurlyeq w \preccurlyeq f_{\alpha}(x) \preccurlyeq x$ for all $x \in C$ and $\alpha \in \mathcal{I}$. Then transitivity of $\preccurlyeq$ implies that $v \preccurlyeq x$ for all $x \in C$. Since $\mathcal{F}$ is concordantly isotone, $f_{\beta}(v) \preccurlyeq f_{\alpha}(x)$ which gives $v \preccurlyeq f_{\alpha}(x)$ for all $x \in C$ and $\alpha \in \mathcal{I}$. Therefore, $C \cup \{v\}$ is a chain in $\mathcal{C}_1(x_0, \mathcal{F},\preccurlyeq)$ but this contradicts the maximality of $C$. Hence, $w$ is a minimal element of the set $\mbox{ComFix}(\mathcal{F}) \cap O_X(x_0)$.` 
\end{proof}

The dual version of Theorem \ref{theorem6} can be stated as follows:
\begin{theorem}\label{theorem7}
Let $(X,d,\preccurlyeq)$ be a preordered $s$-regular $b$-metric space in which a point $x_0$ is fixed and $\mathcal{F}=\{f_{\alpha}\}_{\alpha \in \mathcal{I}}$ be a concordantly isotone family of mappings verifying $f_{\alpha}(x_0) \succcurlyeq x_0$ for all $\alpha \in \mathcal{I}$. Suppose that for any chain $C \in \mathcal{C}_1^*(x_0\mathcal{F},\preccurlyeq)$ there exists $w \in X$ such that $w$ is a common upper bound of the chains $f_{\alpha}(C)$ for each $\alpha \in \mathcal{I}$ and there exists $z \in X$ and $\beta \in \mathcal{I}$ satisfying
\begin{align*}
&f_{\alpha}(w) \succcurlyeq w \succcurlyeq f_{\beta}^i(z),\\
&d(f_{\alpha}^i(w), f_{\beta}^i(z)) \rightarrow 0.
\end{align*}
Then the set $\mbox{ComFix}(\mathcal{F}) \cap O_X^*(x_0)$ is nonempty and contains a maximal element.
\end{theorem}
\begin{theorem}\label{theorem8}
Let $(X,d,\preccurlyeq)$ be a preordered $s$-regular $b$-metric space. Let $\mathcal{F}=\{ f_{\alpha} \}_{\alpha \in \mathcal{I}}$ and $\mathcal{G}=\{g_{\alpha}\}_{\alpha \in \mathcal{I}}$ be a pair of families of self mappings on $X$. Suppose that $\{H_{t,\alpha} \}_{0 \leq t \leq n}$ be a homotopy between $f_{\alpha}$ and $g_{\alpha}$ such that
$$f_{\alpha}=H_{0,\alpha} \preccurlyeq H_{1,\alpha} \succcurlyeq H_{2,\alpha} \preccurlyeq \ldots \succcurlyeq H_{n, \alpha}=g_{\alpha}.$$
Suppose that $x_0 \in \mbox{ComFix}(\mathcal{F})$ and the following conditions are satisfied:

(i) for each $t$, $1 \leq t \leq n$, the family $\mathcal{H}_t=\{H_{t,\alpha}\}_{\alpha \in \mathcal{I}}$ is concordantly isotone,

(ii) for each odd $t$, $1 \leq t \leq n$, for any chain $C \in \mathcal{C}_1^*(\mathcal{H}_t,\preccurlyeq)$ there exists $w \in X$ such that $w$ is a common upper bound of the chains $H_{t,\alpha}(C)$ for all $\alpha \in \mathcal{I}$ and there exists $z \in X$ and $\beta \in \mathcal{I}$ satisfying
\begin{align*}
&H_{t,\alpha}(w) \succcurlyeq w \succcurlyeq H_{t,\beta}^i(z)\\
&d(H_{t,\alpha}^i(w),H_{t,\beta}^i(z)) \rightarrow 0,
\end{align*} 

(iii) for each even $t$, $1 \leq t \leq n$, for any chain $C \in \mathcal{C}_1(\mathcal{H}_t,\preccurlyeq)$ there exists $w' \in X$ such that $w'$ is a common lower bound of the chains $H_{t,\alpha}(C)$ for all $\alpha \in \mathcal{I}$ and there exists $z' \in X$ and $\gamma \in \mathcal{I}$ satisfying
\begin{align*}
&H_{t,\alpha}(w') \preccurlyeq w' \preccurlyeq H_{t,\gamma}^i(z')\\
&d(H_{t,\alpha}^i(w'),H_{t,\gamma}^i(z')) \rightarrow 0,
\end{align*}
Then there exists a fence
$$x_0 \preccurlyeq x_1\succcurlyeq x_2 \preccurlyeq \ldots \succcurlyeq x_n$$
such that for each odd $t$, $1 \leq t \leq n$, $x_t \in \mbox{ComFix}(\mathcal{H}_t) \cap O_X^*(x_{t-1})$ and $x_t$ is a maximal element of the set $\mbox{ComFix}(\mathcal{H}_t) \cap O_X^*(x_{t-1})$ and for each even $t$, $1 \leq t \leq n$, $x_t \in \mbox{ComFix}(\mathcal{H}_t) \cap O_X(x_{t-1})$ and $x_t$ is a minimal element of the set $\mbox{ComFix}(\mathcal{H}_t) \cap O_X(x_{t-1})$.
\end{theorem}
\begin{proof}
Since $x_0 \in \mbox{ComFix}(\mathcal{F})$, $H_{1,\alpha}(x_0) \succcurlyeq H_{0,\alpha}(x_0)=f_{\alpha}(x_0)= x_0$ for all $\alpha \in \mathcal{I}$. This gives $H_{1, \alpha}(x_0)\succcurlyeq x_0$ for all $\alpha \in \mathcal{I}$. Using $\mathcal{C}_1^*(x_0,\mathcal{H}_1,\preccurlyeq) \subset \mathcal{C}_1^*(\mathcal{H}_1,\preccurlyeq)$, (i) and (ii) we deduce that all the conditions of Theorem \ref{theorem7} are satisfied. Therefore, there exists $x_1 \in \mbox{ComFix}(\mathcal{H}_1) \cap O_X^*(x_0)$ such that  $x_1$ is a maximal element of the set $\mbox{ComFix}(\mathcal{H}_1) \cap O_X^*(x_0)$.  Also, $H_{2,\alpha} (x_1)\preccurlyeq H_{1,\alpha}(x_1)=x_1$ for all $\alpha \in \mathcal{I}$ and $\mathcal{C}_1(x_1,\mathcal{H}_2,\preccurlyeq) \subset \mathcal{C}_1(\mathcal{H}_2,\preccurlyeq)$. Using (i) and (iii) we conclude that all the conditions of Theorem \ref{theorem6} are satisfied. Therefore, there exists $x_2 \in \mbox{ComFix}(\mathcal{H}_2) \cap O_X(x_1)$ such that $x_2$ is a minimal element of the set $\mbox{ComFix}(\mathcal{H}_2) \cap O_X(x_1)$. Proceeding likewise we get a fence
$$x_0 \preccurlyeq x_1 \succcurlyeq x_2 \preccurlyeq \ldots \succcurlyeq x_n,$$
where for each odd $t$, $1 \leq t \leq n$, $x_t \in \mbox{ComFix}(\mathcal{H}_t) \cap O_X^*(x_{t-1})$ and $x_t$ is a maximal element of the set $\mbox{ComFix}(\mathcal{H}_t) \cap O_X^*(x_{t-1})$ and for each even $t$, $1 \leq t \leq n$, $x_t \in \mbox{ComFix}(\mathcal{H}_t) \cap O_X(x_{t-1})$ and $x_t$ is a minimal element of the set $\mbox{ComFix}(\mathcal{H}_t) \cap O_X(x_{t-1})$.
\end{proof}
\begin{remark}
\emph{It is observed that Theorems \ref{theorem6}-\ref{theorem8} hold if we take $(X,d,\preccurlyeq)$ to be a preordered regular metric space.} 
\end{remark}
\section*{Acknowledgements}
 $^*$corresponding author\\
 The corresponding author is supported by UGC Non-NET fellowship (Ref.No. Sch/139/Non-NET/  Math./Ph.D./2017-18/1028).

\textbf{Anuradha Gupta}\\
 Department of Mathematics, Delhi College of Arts and Commerce,\\
  University of Delhi, Netaji Nagar, \\
  New Delhi-110023, India.\\
  \vspace{0.2cm}
  email: dishna2@yahoo.in\\
  \textbf{Manu Rohilla}\\
  Department of Mathematics, University of Delhi, \\
  New Delhi-110007, India.\\
  email: manurohilla25994@gmail.com

\begin{thebibliography}{99}
\bibitem{11} A. V. Arutyunov, E. S. Zhukovski\u{\i}\ and\ S. E. Zhukovski\u{\i}, On the coincidence points of mappings in partially ordered spaces, Dokl. Math. {\bf 88} (2013), no.~3, 710--713.

\bibitem{1} A. V. Arutyunov, E. S. Zhukovskiy\ and\ S. E. Zhukovskiy, Coincidence points principle for mappings in partially ordered spaces, Topology Appl. {\bf 179} (2015), 13--33.

\bibitem{2} U. Y. Batsari, P. Kumam\ and\ S. Dhompongsa, Fixed points of terminating mappings in partial metric spaces, J. Fixed Point Theory Appl. {\bf 21} (2019), no.~1, Art. 39, 20 pp.

\bibitem{3} I. A. Bakhtin, The contraction mapping principle in almost metric space, {Functional analysis, No. 30 (Russian)}, 26--37, Ul$^\prime$yanovsk. Gos. Ped. Inst., Ul$^\prime$yanovsk, 1989.

\bibitem{4} S. Czerwik, Contraction mappings in $b$-metric spaces, Acta Math. Inform. Univ. Ostraviensis {\bf 1} (1993), 5--11.

\bibitem{5} T. N. Fomenko\ and\ D. A. Podoprikhin, Common fixed points and coincidences of mapping families on partially ordered sets, Topology Appl. {\bf 221} (2017), 275--285.

\bibitem{6} T. Fomenko\ and\ D. Podoprikhin, On preservation of common fixed points and coincidences under a homotopy of mapping families of ordered sets, J. Optim. Theory Appl. {\bf 180} (2019), no.~1, 34--47.

\bibitem{7} T. Kamihigashi\ and\ J. Stachurski, Simple fixed point results for order-preserving self-maps and applications to nonlinear Markov operators, Fixed Point Theory Appl. {\bf 2013}, 2013:351, 10 pp.

\bibitem{8} D. A. Podoprikhin\ and\ T. N. Fomenko, Preservation of the fixed point property and the coincidence property under homotopy of the mappings of ordered sets, Dokl. Math. {\bf 96} (2017), no.~3, 591--593.

\bibitem{9} D. A. Podoprikhin\ and\ T. N. Fomenko, On coincidences of families of mappings of ordered sets, Dokl. Math. {\bf 94} (2016), no.~3, 620--622. 

\bibitem{10} J. W. Walker, Isotone relations and the fixed point property for posets, Discrete Math. {\bf 48} (1984), no.~2-3, 275--288. 

\end{thebibliography}
\end{document}